\documentclass[12pt,a4paper]{amsart}
\DeclareMathOperator*{\esssup}{ess\,sup}

\usepackage{graphicx}
\usepackage{tikz}
\usetikzlibrary{decorations.fractals}

\usepackage{graphicx,color}

\usepackage{amsmath,amsthm}
\usepackage{amssymb,latexsym}
\usepackage{enumerate}
\usepackage{amscd,amssymb,amsthm}
\usepackage{graphicx}

\usepackage{mathrsfs,amssymb}

\usepackage{amscd,amssymb,amsthm}
\usepackage{graphicx}

\usepackage{amsmath,amsthm}
\usepackage{amssymb,latexsym}
\usepackage{enumerate}
\usepackage{amscd,amssymb,amsthm}
\usepackage{graphicx}

\usepackage{xcolor}

\usepackage{textcomp,stmaryrd}

\usepackage{appendix}

\usepackage{xcolor}

\newtheorem{theorem}{Theorem}[section]

\newtheorem{corollary}[theorem]{Corollary}
\newtheorem{lemma}[theorem]{Lemma}
\newtheorem{proposition}[theorem]{Proposition}

\theoremstyle{definition}
\newtheorem{definition}{Definition} 
\newtheorem{remark}[theorem]{Remark}

\numberwithin{equation}{section}

\newcommand{\dx}{\,\mathrm d }

\newcommand{\Om}{\Omega}

\newcommand{\reale}{\mathbb{R}}
\newcommand\rN{\reale^N}

\DeclareMathOperator{\divergenza}{div}

\newcommand{\spp}{ u_n }

\newcommand{\dual}{L^{p^\prime}\left(0,T,W^{-1,p^\prime}(\Omega)\right)}
\newcommand{\pspace}{L^p\left(0,T,W^{1,p}_0(\Omega)\right)}
\newcommand{\parspbound}{L^\infty\left(0,T,L^2(\Omega)\right)}
\newcommand{\parspcont}{C^0\left([0,T],L^2(\Omega)\right)}

\newcommand{\pspacegen}{L^p\left(0,T,X\right)}
\newcommand{\pspacegeninfty}{L^\infty\left(0,T,X\right)}
\newcommand{\parspgencont}{C^0\left([0,T],X\right)}

\renewcommand{\le}{\leqslant}
\renewcommand{\ge}{\geqslant}

\def\Xint#1{\mathchoice
{\XXint\displaystyle\textstyle{#1}}%
{\XXint\textstyle\scriptstyle{#1}}%
{\XXint\scriptstyle
\scriptscriptstyle{#1}}%
{\XXint\scriptscriptstyle
\scriptscriptstyle{#1}}%
\!\int}
\def\XXint#1#2#3{{
\setbox0=\hbox{$#1{#2#3}{\int}$}
\vcenter{\hbox{$#2#3$}}\kern-.5\wd0}}
\def\dashint{\Xint-}

\newcommand{\average}[3]{ \dashint_{#1} #2 \dx{#3}} 

\begin{document}
\title[Evolution problems with singular coefficients]{Nonlinear evolution problems \\ with singular coefficients  
in the lower order terms}


\author{Fernando Farroni}
\address{Fernando Farroni \\ Dipartimento di Matematica e Applicazioni R. Caccioppoli \\ Universit\`{a} 
degli Studi  di Napoli Federico II \\ Complesso Monte S. Angelo, via Cinthia \\ I-80126 Napoli, Italy}
\email{fernando.farroni@unina.it}
\author{Luigi Greco}
\address{Luigi Greco \\ Dipartimento di Ingegneria elettrica e delle Tecnologie dell'Informazione \\ Universit\`{a} 
degli Studi  di Napoli Federico II \\ Piazzale Tecchio, 80 \\ I-80126 Napoli, Italy}
\email{luigreco@unina.it}
\author{Gioconda Moscariello}
\address{Gioconda Moscariello \\ Dipartimento di Matematica e Applicazioni R. Caccioppoli \\ Universit\`{a} degli Studi di Napoli Federico II \\ Complesso Monte S. Angelo, via Cinthia \\ I-80126 Napoli, Italy}
\email{gmoscari@unina.it}
\author{Gabriella Zecca}
\address{Gabriella Zecca \\ Dipartimento di Matematica e Applicazioni R. Caccioppoli \\ Universit\`{a} 
degli Studi  di Napoli Federico II \\ Complesso Monte S. Angelo, via Cinthia \\ I-80126 Napoli, Italy}
\email{g.zecca@unina.it}


\thanks{The authors are members of Gruppo Nazionale per l'Analisi Matematica, la Probabilit\`a e le loro Applicazioni (GNAMPA) of INdAM. The research of G.M. has been partially supported by the National Research Project PRIN \lq\lq Gradient flows, Optimal Transport and Metric Measure Structures'', code~2017TEXA3H}
 
\begin{abstract}
We consider a Cauchy--Dirichlet problem for a quasilinear second order parabolic equation with lower order term 
driven by a singular coefficient. 
We establish an existence result to such a problem and we   describe
 the time behavior of the solution in the case of the infinite–time horizon.
\end{abstract}

\date{\today}

\maketitle

\tableofcontents
\addtocontents{toc}{\protect\thispagestyle{empty}}
\pagenumbering{gobble}

%


\section{Introduction}
The aim of this paper is to study 
the  following 
Cauchy--Dirichlet problem 
\begin{equation}\label{mp}
\left\{
\begin{array}{rl}
&  
u_t     - 
\divergenza  
A(x,t,u, \nabla u)  
 = 
f
\qquad\text{in $  \Omega_{T }$},
  \\
\, \\
&   u  = 0    
\qquad\text{on $\partial \Omega    \times (0,T)  $}
, \\
\, \\
& u (\cdot,0) = u_0     
\qquad\text{in $  \Omega    $}. \\
\end{array}
\right.
\end{equation}
Here  $\Omega$
is  a bounded open 
subset  of $\mathbb R^N$, with $N \ge2$. Correspondingly,   $\Omega_T:=\Omega \times (0,T)$ is 
the parabolic cylinder over $\Omega$ of height $T>0$.  We adopt a usual notation $   u _t$  for 
the time derivative and  $\nabla u$ for gradient   with respect to the space variable. 
We let $  {2N}/{(N+2)}<p < N$. 
For the data related to problem \eqref{mp}, we consider
\begin{align}
& u_0\in L^2 (\Omega)
\\
& f \in \dual
\end{align}

\medskip

We assume that 
$$A=A(x,t,u,\xi) \colon \Omega_T \times \mathbb R \times \rN \rightarrow \rN$$ is a 
Carath\'eodory function (i.e. $A$ is measurable w.r.t. $(x,t)\in \Omega_T$ for all $(u,\xi)\in \mathbb R \times \rN$ 
and continuous w.r.t. $(u,\xi)\in \mathbb R\times \rN$ for a.e.
$(x,t)\in \Omega_T$) satisfying 
the following monotonicity and boundedness conditions 
\begin{align}
& \label{1.3}
A(x,t,u, \xi)\cdot \xi \ge \alpha |\xi|^p - \left(
b(x,t) |u|
\right)^p - H(x,t)
\\
&\label{1.4}  \left[ A(x,t,u, \xi)
 - A(x,t,u, \eta) \right] \cdot (\xi-\eta)
>0  \qquad \text{if $\xi \neq \eta$} 
\\
& \label{1.5} \left|A(x,t, u,\xi)\right|  \le \beta |\xi|^{p-1} +  \left(
b(x,t) |u|
\right)^{p-1} + K(x,t)
\end{align}
for a.e.  $(x,t)\in \Omega_T$ and for any $u\in \mathbb R$ and $\xi,\eta \in \rN$.
Here  $\alpha,\beta$ are positive constants, while 
$H$, $K$ and $b$ are nonnegative measurable functions defined on $\Om_T$
such that
$H \in L^1(\Omega_T)$,
$ K\in L^{p^\prime}(\Omega_T)$
and 
\begin{equation}\label{b}
b 
\in 
L^\infty 
\left(
0,T, L^{N,\infty} (\Omega)
\right)
\end{equation} 
Here $L^{N,\infty} (\Omega)$ denotes the Marcinkiewicz space (see Section \ref{funsps} for the definition).

\medskip

Taking into account all the assumptions above, we consider the following notion of solution.
\begin{definition}\label{defsolu}
A solution to problem \eqref{mp} 
is a function 
\[
u \in    \parspcont  \cap \pspace
\]
such that for a.e. $t\in (0,T)$ we have
\begin{equation}\label{ds}
-\int_0^t \int_\Omega u \varphi_t \,dxds
+
\int_0^t \int_\Omega A(x,s,u,\nabla u)\cdot \nabla \varphi \,dxds
=\int_\Omega u_0 \varphi (x,0)\,\dx x
+
\int_0^t 
 \langle f,  \varphi \rangle 
\,\dx s
\end{equation} 
for every $\varphi \in C^\infty(\bar \Omega_T)$ such that ${\rm supp }\,\varphi \subset [0,T)\times \Omega$.

Moreover, if  $ u \in C_{\rm loc}^0 
\left(\left[0,\infty\right),L^2(\Om)
\right)\cap L_{\rm loc}^p\left(
0,\infty,W^{1,p}_0 (\Omega)
\right)$
and the above holds true for all $T>0$, then $u$ is called a solution to 
problem \eqref{mp} in $\Om \times (0,\infty)$.
\end{definition}
For the notation related to parabolic type function spaces such as $\pspace$ or similar, we refer the reader to Section \ref{pre} below. Above and
throughout the paper, for two vectors $\xi,\eta\in\mathbb R^N$ we denote by
 $\xi\cdot\eta$ their scalar 
product and  we denote by $\langle\cdot,\cdot\rangle$  
the duality between $W^{-1,p^\prime}(\Omega)$ and $W_0^{1,p}(\Omega)$. 

\medskip

Model equation that 
we consider in this context is 
\begin{equation}\label{1.9}
u_t     - 
\divergenza   \left[
|\nabla u |^{p-2} \nabla u  + 
|u|^{p-2}u
\left(
\mu \frac {h(t)} {|x|} + b_0(x,t)    \right)    \frac x {|x|}
\right]
 = 
f
\end{equation} 
where 
$\mu >0$, $h \in L^\infty(0,T)$ and 
%
%
$b_0 \in L^\infty(\Om_T)$.

\medskip

When $p=2$, 
the linear homogeneous equation 
in
\eqref{mp} 
describes the evolution of some Brownian motion and it is also known as 
Fokker--Planck equation
(see e.g. \cite{CLLP,Porretta} and the references therein).    
We remark that the boundedness of the 
growth coefficient $b(x,t)$
is too restrictive in many applications as, for instance, in the case 
of the diffusion model for 
semiconductor devices
(see e.g. \cite{Fang Ito}).
 On the other hand, a low
  integrability assumption 
 in $\Om_T$ for the  term $b=b(x,t)$ 
 does not guarantee the existence of a distributional solution in the sense of definition \eqref{ds}. In this case, other definitions of solutions have been introduced
 (see e.g. \cite{
Porretta}). 
Assumption \eqref{b}, in view of 
Sobolev embedding theorem (see Theorem \ref{lorentz} below), 
guarantees that 
\[
\int_{\Om_T} A(x,t,u,\nabla u )\cdot \nabla u \dx x \dx t < \infty
\] 
that is,  a solution in the sense of Definition 
\ref{defsolu}
has  finite energy.

\medskip

Our existence result reads as follows.  
\begin{theorem} \label{ttt}
Let assumptions 
\eqref{1.3},
\eqref{1.4},   
\eqref{1.5} and \eqref{b}
be in charge.   
Assume further that 
\begin{equation}\label{dist}
\mathscr D_b :=
\esssup_{0<t<T } \,\,{\rm dist}_{L^{N,	\infty}(\Omega)} \left(
b,{L^{ \infty}(\Omega)}
\right)
<
 \frac {\alpha^{1/p}}{S_{N,p}} 
\end{equation}
Then
problem  \eqref{mp} admits  a solution.  
\end{theorem}
The constant $S_{N,p}$ is the one of 
Sobolev embedding theorem in Lorentz spaces (see 
Theorem 
\ref{lorentz} below). In \eqref{dist}  
${\rm dist}_{L^{N,	\infty}(\Omega)} \left(
b,{L^{ \infty}(\Omega)}
\right)$
denotes the distance from bounded functions of the function $b$  with respect to the $L^{N,\infty}$--norm (see formula \eqref{distbis}  below
for   the   definition).

\medskip

Condition \eqref{dist}
, for the first time   introduced 
 in \cite{GGM}
and in \cite{GMZ2},     
does not 
imply a smallness condition on the norm.
Indeed, in the example 
\eqref{1.9} it just gives a bound on the constant $\mu$.
In particular, 
\eqref{dist} 
holds true 
whenever $b\in L^\infty 
\left(
0,T ,  L^{N,q} (\Omega)
\right)
 $
for $1\le q < \infty$.   
Here $L^{N,q} (\Omega)$ denotes the Lorentz space
(see Section \ref{funsps} for the definition).
We don't know if condition \eqref{dist}
is optimal in our framework. Nevertheless, 
in  the elliptic counterpart of Theorem  \ref{ttt} (considered in  \cite{FMGZ})
such a condition turns to be optimal at least for $p=2$.  
%
%
%
%
%
%

\medskip

We also study the behavior on time of a weak solution 
given in Theorem \ref{ttt}.
More precisely, we estimate on time the $L^2$--norm of $u$ with the solution of a Cauchy problem related to a o.d.e.   (see Theorem \ref{asyt1}). As  consequence we provide estimates that highlight the different decay behavior   as the exponent $p$ varies when $T=\infty$. The presence of the lower order term does not affect the decay to zero of the $L^2$-norm as $T$ goes to infinity (Corollary \ref{finalcor} below). Recent results about the time decay for solutions to   parabolic problems in absence of the lower order term can be found in \cite{Moscariello-Frankowska, 
Moscariello-Porzio,Por}
and the references therein.

\medskip

The novelty of the paper lies on the fact that 
 in Theorem  \ref{ttt}  above
and Theorem \ref{asyt1} below we consider a family of operators  not coercive with 
a singular growth coefficient in the lower order term. We recall that bounded functions are not dense in the 
Marcinkiewicz space $L^{N,\infty}(\Omega)$. In order to find a  solution to   \eqref{mp}, 
the 
main tool is  an apriori estimate that could have interest by itself (see Proposition \ref{mprop}).  Thanks to  Leray--Schauder fixed point theorem,  we first solve the problem  when $b(x,t)$ is bounded. 
%
%
%
Then,  we obtain a solution to 
 \eqref{ttt} as a limit of a sequence of solutions to   suitable  approximating problems. 
A  solution of Theorem \ref{ttt} satisfies an energy equality and then by using a recent  result of \cite{Moscariello-Frankowska} we are able to
describe its asymptotic behavior.
\section{Preliminary results}\label{pre}
\subsection{Basic notation}
We adopt the usual symbols $\lesssim$ and $\gtrsim$ for inequalities which holds true up to not influent constants (in our case   typically depending   on $N$, $p$,
$\alpha$ and $\beta$). 

We will   denote by $ C$ (or by similar symbols such as $C_1, C_2,\dots$)  
a generic positive constant, which may possibly
vary from line to line. 
The dependence of $C$ upon various parameters     will be highlighted 
 in parentheses, adopting a notation of the type $C(\cdot,\dots,\cdot)$.

\subsection{Function spaces}  \label{funsps}
Let $\Om$ be a bounded domain in $\mathbb R^N$.
Given $1<p<\infty$ and $1\leqslant q<\infty$, the Lorentz space $L^{p,q}(\Om)$ consists of all measurable functions $f$ defined on $\Omega$ for which the quantity
\begin{equation}\label{norma Lorentz}
\|f\|_{p,q}^q:=p\int_{0}^{\infty} \left[\lambda_f(k)\right]^{\frac{q}{p}}k^{q-1}\dx k
\end{equation}
is finite, where $\lambda_f(k):=  \left| \left\{ x\in \Om: |f(x)|>k \right\}  \right|$
is the distribution function of $f$. Note that $\|\cdot\|_{p,q}$  is equivalent to a norm and $L^{p,q}(\Om)$ becomes a Banach space when endowed with it (see \cite{bs,O}). For $p=q$, the Lorentz space
$L^{p,p}(\Omega)$ reduces to the Lebesgue space $L^p(\Omega)$. For $q=\infty$, the class $L^{p,\infty}(\Omega)$ consists of all measurable functions $f$ defined on $\Omega$ such that
\begin{equation*}\label{2.1}
\|f\|^p_{p,\infty}:=\sup_{k>0}k^{ p} \lambda_f(k)< \infty
\end{equation*}
and it coincides with the Marcinkiewicz class, weak-$L^p(\Omega)$.

For Lorentz spaces the following inclusions hold
\[
L^r (\Om)\subset L^{p,q}(\Om)\subset  L^{p,r} (\Om) \subset L^{p,\infty}(\Om)\subset L^q(\Om),
\]
whenever $1\leqslant q<p<r\leqslant \infty.$ Moreover, for $1<p<\infty$, $1\leqslant q\leqslant \infty$ and $\frac 1 p+\frac 1 {p'}=1$, $\frac 1 q+\frac 1 {q'}=1$, if $f\in L^{p,q}(\Omega)$, $g\in L^{p',q'}(\Omega)$ we have the H\"{o}lder--type inequality
\begin{equation}\label{holder}
\int_{\Omega}|f(x)g(x)|\dx x \leqslant \|f\|_{p,q}\|g\|_{p',q'}.
\end{equation}

As it is well known, $L^\infty(\Om)$ is not dense in $L^{p,\infty}(\Om)$. 
For a function $f \in 
L^{p,\infty} (\Om)$ we define
\begin{equation}\label{distbis}
{\rm dist}_{L^{p,\infty} (\Om) } (f,L^\infty (\Om))  =\inf_{g\in L^\infty(\Om)} \|f-g\|_{L^{p,\infty}(\Om)}.
\end{equation}
In order to characterize the distance in \eqref{dist}, 
we introduce for all  $k>0$ the truncation operator at heights $k$, namely
\begin{equation*}\label{201206272}
\mathcal T_k y :=\frac{y}{|y|}\min\{|y|, k\} \qquad \text{for $y \in \mathbb R\,$}.
\end{equation*}
It is easy to verify that
\begin{equation}\label{distlim}
\lim_{k\to\infty}\|f-\mathcal T_kf\|_{p,\infty} 
=
{\rm dist}_{L^{p,\infty} (\Om) } (f,L^\infty (\Om))\,.
\end{equation}
Clearly, for $1\leqslant q<\infty$ 
any function in $L^{p,q}(\Omega)$ has  vanishing distance 
to $L^\infty(\Omega)$. Indeed, $L^\infty(\Omega)$ is dense in $L^{p,q}(\Omega)$, the latter being continuously embedded into $L^{p,\infty}(\Omega)$. 
\par
Assuming that $0\in\Omega$, a typical  element of $L^{N,\infty}(\Om)$ is $b(x)=B/|x|$, with $B$ a positive constant. An elementary calculation shows that
\begin{equation}\label{esempio-dist}
{\rm dist}_{L^{N,\infty} (\Om) } (b,L^\infty (\Om))=B\,\omega_N^{1/N}
\end{equation}
where $\omega_N$ stands for the Lebesgue measure of the unit ball of $\rN$.  

\medskip

The Sobolev embedding theorem in Lorentz spaces \cite{O,A} 
reads as
\begin{theorem}\label{lorentz}
Let us assume that $1<p<N$, $1\leqslant q\leqslant p$, then every function $u\in W_0^{1,1}(\Omega)$ verifying $|\nabla u|\in L^{p,q}(\Om)$ actually belongs to $L^{p^*,q}(\Om)$, where $p^*:=\frac{Np}{N-p}$ is the Sobolev exponent of $p$
and
\begin{equation}\label{Sobolevin}
\|u\|_{p^*,q}\leqslant S_{N,p}\|\nabla u\|_{p,q}
\end{equation}
where $S_{N,p}$ is the Sobolev constant given by $S_{N,p}=\omega_N^{-1/N}
    \frac{p }{ N-p  }  $.
\end{theorem}

\medskip

For our purposes, we also need to introduce some spaces involving the time variable. Hereafter, 
for the time derivative
$u_t$ of a function $u$
we adopt 
the alternative notation 
 $\partial_t u$, $\dot{u}$, $u^\prime$ or 
$\dx u/\dx t$.
Let $T>0$. 
If we let $X$ be a separable Banach space endowed with a norm $\|\cdot\|_X$, the space 
$
\pspacegen
$ 
is defined as the class of all measurable
functions $u\colon [0,T] \rightarrow X$
such that  
\[
\|u\|_{\pspacegen}
:=\left(
\int_0^T \|u(t)\|^p_X \dx t
\right)^{1/p}
<\infty
\]
whenever $1\le p <\infty$, and
\[
\|u\|_{\pspacegeninfty}
:=\esssup_{0<t<T }  
 \|u(t)\| _X  
<\infty
\]
for $p=\infty$.
Similarly, the space 
$
\parspgencont
$ represents the class of all 
 continuous 
functions $u\colon [0,T] \rightarrow X$
such that  
\[
\|u\|_{\parspgencont}
:=\max_{0\le t \le T}
 \|u(t)\| _X  
<\infty
\]

\medskip 

We
will mainly deal with the case where $X$ is either a Sobolev space or a Lorentz space. In particular, we recall a
well known result (see e.g. \cite[Proposition  1.2, Chapter III, pag. 106]{Showalter book}) involving  the class of functions $W_p(0,T)$ defined 
as follows
\[
W_p(0,T):=
\left\{
v \in \pspace\colon\, v_t \in {\dual}
\right\}
\]
equipped with the norm
\[
\|
u
\|_{W_p(0,T)}
 := 
\|
u
\|_{\pspace}
+
\|
u_t
\|_{\dual}
\]
%
\begin{lemma}\label{incl}
Let $p > 2N /(N+2)$.  Then 
$W_p(0,T)$ is contained into
$
\parspcont$ and any function $u \in W_p(0,T)$ satisfies
\[
\|
u
\|_{\parspcont}
\le C
\|
u
\|_{W_p(0,T)}
\]
for some constant $C>0$. \par
Furthermore, 
the function $t\in[0,T]\mapsto\|u(\cdot,t)\|^2_{L^2(\Omega)}$ is absolutely continuous and \begin{equation}\label{tid} \frac 1 2 \frac {\rm d} {\dx t} \|u(\cdot,t)\|^2_{L^2(\Omega)} = \left\langle
u_t(\cdot,t),u(\cdot,t) \right\rangle \qquad \text{for a.e. $t\in[0,T]$}
 \end{equation}
\end{lemma}
Finally, we recall the classical compactness result due to Aubin--Lions 
(see e.g. \cite[Proposition  1.3, Chapter III, pag. 106]{Showalter book}).
\begin{lemma}[Aubin--Lions]
Let $X_0, X, X_1$ be Banach spaces with $X_0$ and $X_1$ reflexive. Assume that $X_0$ is compactly embedded into $X$ and $X$ is continuosly embedded into $X_1$. For $1<p,q<\infty$ let $$W:= \{u \in L^p(0,T,X_0)\colon \partial_t u  \in L^q(0,T,X_1)\}$$ Then $W$ is compactly embedded into $L^p(0,T,X)$.
\end{lemma}
A prototypical example of application of this lemma corresponds to the 
choices $q=p$, $X_0=W^{1,p}_0(\Om)$, $X_1=W^{-1,p^\prime} (\Om)$ and $X=L^p(\Om)$ if $p\ge 2$ or $X=L^2(\Om)$ for $\frac {2N}{N+2}<p<2$. Obviously 
  $L^2(\Om)\subset L^p(\Om)$ as long as $p<2$, and therefore we deduce the following.
\begin{lemma} 
If $p>\frac {2N}{N+2}$ then $W_p(0,T)$  is  compactly embedded into  $L^p(\Om_T)$. 
\end{lemma}

\subsection{Gronwall type results}
We recall a couple of lemmas  (see  \cite{Moscariello-Frankowska})
whose application  will 
be essential in the study of the time 
behavior of the solutions to \eqref{mp}.
\begin{lemma}\label{lCau1}
Consider  a Carath\'eodory function  $\psi : [t_0,T ] \times \mathbb R   \to \mathbb R_+$  such that  for every $r >0$ there exists $k_r \in L^1(t_0,T ; \mathbb R_+)$ satisfying for a.e. $t \in [t_0,T ]$
$$ \sup_{|x| \leq r } \psi(t,x) \leq k_r(t).$$

 Let  $g \in L^1(t_0,T ; \mathbb R)$ and $\gamma: [t_0,T ] \to \mathbb R_+  $ be measurable, bounded and  satisfy
\begin{equation}\label{GM2} \gamma(t_2) - \gamma(t_1) + \int_{t_1}^{t_2}\psi (t,\gamma(t))dt \leq  \int_{t_1}^{t_2} g(t) dt \qquad t_0 \leq t_1 \leq t_2 \leq T  .\end{equation}

Then there exists a  solution $x(\cdot) \in W^{1,1}([t_0,T])$ of
the Cauchy problem
\begin{equation}\label{GM1} \left\{\begin{array}{lll}
x'(t) = - \psi(t,x(t)) +g(t), \quad {\rm a.e. } \\
x(t_0)= \gamma(t_0)
\end{array}
\right.
\end{equation}
 such that $\gamma(t) \leq x(t)$ for all $t \in [t_0,T  ].$

Furthermore, if  $g \in L^1(t_0,\infty; \mathbb R)$,   $\psi$ is defined on $[t_0,+\infty) \times \mathbb R  $, $\gamma : [t_0,+\infty) \to \mathbb R_+ $
 is measurable and locally bounded and  for every $T>t_0$ the above assumptions hold true,   then there exists a solution $x $ to (\ref{GM1}) defined
  on $[t_0,\infty)$ such that $\gamma \leq x$. In particular,
 $$\limsup_{t \to \infty} \gamma(t) \leq \limsup_{t \to \infty} x(t). $$ 
\end{lemma}
%
%
%

\begin{lemma}\label{lCau2}
Under all the assumptions of  Lemma \ref{lCau1} suppose that $\psi (t,a)=0$ for all $a \leq 0$, $ g (\cdot) \geq 0$,  $\psi (t,\cdot)$ is increasing for a.e. $t \in [t_0,T  ] $ and
that for any $R>r >0$ there exists $\bar k_{R,r} \in L^1(t_0,T ; \mathbb R_+)$ satisfying for a.e. $t \in [t_0,T ]$
$$  |\psi(t,x) - \psi(t,y)| \leq k_{R,r}(t)|x - y| \quad \forall\; x,y \in [r,R]  .
 $$
Then the solution $z(\cdot)$ of
\begin{equation}\label{GM5}  \left\{\begin{array}{lll}
z'(t) = - \psi(t,z(t)), \quad {\rm a.e. } \\
z(t_0)= \gamma(t_0)
\end{array}
\right.
\end{equation}  is unique and  well defined on $ [t_0,T  ]$,   $z(\cdot) \geq 0$ and  $\gamma(t) \leq x(t) \leq z(t) + \int_{t_0}^{t}g(s)ds$  for all $t \in [t_0,T  ]$, where $x(\cdot)$ is as in the claim of Lemma \ref{lCau1}.
%
%
%
\end{lemma}

\section{Weak type and a priori estimates for an auxiliary problem}  
An a priori estimate on the distribution
function of a solution to problem \eqref{mp}
will be fundamental in order to prove Theorem  \ref{ttt}.  
%
%
%
%
To this aim, we let $\phi \colon  \mathbb R \rightarrow \mathbb R  $ be the function defined as
\[
\phi(w):= \frac 1 {p-1} \left[
1-\frac 1 {(1+|w|)^{p-1}}
\right]
{\rm sign}(w)
\,\, \text{for $w \in \mathbb R \setminus\{0\}$ and $\phi(0):=0$}
\]
We   set 
$\Phi(w):=\int_0^{|w|} \phi(\rho)\,d\rho$ and   $\Psi \colon (0,\infty) \rightarrow (0,\infty)$  be the reciprocal of the restriction of $\Phi$ to $(0,\infty)$, so that 
it is a continuous and decreasing function
 such that
$\Psi(k) \rightarrow 0$ as $k \rightarrow \infty$. 
With this notation at hand, 
we have the following result.
\begin{lemma}\label{ldlam}
Let assumptions 
\eqref{1.3} 
and
\eqref{1.5}
be in charge 
and $b \in L^\infty\left(0,T,L^{N,\infty}(\Omega)\right)$. 
For  a fixed 
$\lambda \in (0,1]$, assume that the problem
\begin{equation}\label{mp3}
\left\{
\begin{array}{rl}
&  \displaystyle{
\frac {u_t} \lambda   }  - 
\divergenza  A\left(x,t,u, \frac {\nabla u} \lambda\right)
 = 
f
\qquad\text{in $  \Omega_{T }$},
  \\
\, \\
&   u  = 0    
\qquad\text{on $\partial \Omega    \times (0,T)  $}
, \\
\, \\
& u (\cdot,0) = \lambda     u_0  
\qquad\text{in $  \Omega    $}
, \\
\end{array}
\right.
\end{equation}
admits a solution $u \in \parspcont \cap \pspace $. Then,
for every $k>0$ we have    
\begin{equation}\label{decay}
\sup_{0<t<T}
\left|
\left\{
x\in \Omega\colon\, |u(x,t)| >k
\right\}
\right|\le 
\Psi( k)M_0 
\end{equation}
where
\begin{equation}\label{K0}
\begin{split}
M_0 =
\frac 1 2 \|
u_0
\|^2_{L^2(\Omega)}
+
\|
H
\| _{L^1(\Omega_T)}
 +
 \|b\|^p_{L^p(\Omega_T)}+ \alpha^{-1/(p-1)}
\|
f
\|^{p^\prime}_{\dual}
%
%
%
\end{split}
\end{equation}
\end{lemma}
\begin{proof}
First of all, we set $w:= u/\lambda$, in such a way that $w$ solves the problem
\begin{equation}\label{mp4}
\left\{
\begin{array}{rl}
&  
w_t       - 
\divergenza  A\left(x,t,\lambda w ,  \nabla w  \right)
 = 
f
\qquad\text{in $  \Omega_{T }$},
  \\
\, \\
&   w  = 0    
\qquad\text{on $\partial \Omega    \times (0,T)  $}
, \\
\, \\
& w (\cdot,0) =    u_0  
\qquad\text{in $  \Omega    $}
, \\
\end{array}
\right.
\end{equation}
We fix $t \in (0,T)$  and we  choose   $ \varphi:=\phi(w)\chi_{(0,t)} $
as a test   function for 
\eqref{mp4}.
(This can be done since $\phi(\cdot)$ is a Lipschitz function in the whole of $\mathbb R$.)
In this way,  we have
\begin{equation}\label{test1}
\begin{split}
&
\int_\Omega
\Phi(w)\,\dx x
+
\int_{\Omega_{t}}
 A(x,s, \lambda  w,\nabla w) \cdot \nabla \phi(w)
\,
\dx x \dx s
=
\int_\Omega
\Phi(u_0)\,\dx x
+
\int_0^{t} 
\left\langle
f
,
  \phi(w)
\right\rangle
\dx s 
\end{split}
\end{equation}
Observe explicitly that
\[
 \nabla \phi(w) = \frac { \nabla w }{(1+|w|)^p}  \chi_{(0,t)}  
\]
Therefore, by 
\eqref{1.3}
we have
\begin{equation}\label{test2}
\begin{split}
& 
\int_\Omega \Phi(w)       \,\dx x
+\alpha
\int_{\Omega_t}
\frac
{
|\nabla w|^p 
}
{(1+|w|)^p}\,\dx x \dx s
\\
&
\le
\int_\Omega \Phi(u_0)       \,\dx x
+ \int_0^{t} 
\left\langle
f,
  \phi(w) 
\right\rangle
\dx s 
+
\int_{\Omega_t}
\left(
\frac{   b  |w| }{1+|w|}
\right)^p
\,\dx x \dx s
+
\int_{\Omega_t}
 {H(x,s)}
 \dx x \dx s
\end{split}
\end{equation}
By means of Young inequality 
we have 
\begin{equation}
\begin{split}
 \int_0^{t} 
\left\langle
f
,
  \phi(w) 
\right\rangle
\,
\dx s 
&\le 
 \alpha
\int_{\Omega_t}
\frac{|\nabla w|^p}
{
(1+|w|)^p
}
\,\dx x \dx s
+
  \alpha ^{-1/(p-1)}  \|f\|^{p^\prime}_{\dual} 
\end{split}
\end{equation}
In view of the latter estimate and of the fact that $0\le \Phi(k) \le \frac 1 2  k ^2$ for any $k\ge 0$,  from \eqref{test2}
we infer
\begin{equation}\label{test22}
\begin{split}
& 
\int_\Omega \Phi(w)       \,\dx x
\le M_0 
\end{split}
\end{equation}
For $t \in (0,T)$ fixed,  we set $E_k(t):=\{x\in \Omega\colon\, |u(x,t)|>k\}$ for $k>0$. Since $\lambda \in [0,1]$, clearly $E_k(t) \subset \{x\in \Omega\colon\, |w(x,t)|>k\}$. 
%
By  the monotonicity of $\Phi$
and by \eqref{test22}, we have
 $ |E_k(t)|\Phi(k) \le M_0   $, which implies the claimed estimate.
\end{proof}

Now, we are in position to prove some a priori estimate for a solution to
 problem 
\eqref{mp3}.
\begin{proposition}\label{mprop}
Let
assumptions of 
Lemma \ref{ldlam}
%
%
%
be fulfilled and assume that
condition
\eqref{dist} holds true.
%
%
%
For  a fixed 
$\lambda \in (0,1]$, 
any solution $u \in \parspcont \cap \pspace $
to problem
\eqref{mp3}
%
%
%
satisfies   
the following estimate
\begin{equation}\label{apr1bis}
\begin{split}
\sup_{0<t<T}
 \int_\Omega |u  (\cdot,t) |^2\, \dx x
+ 
\int_{\Omega_T}
| \nabla u  |^p
\, \dx x \dx t
 \le
C  
\end{split}
\end{equation}
for some positive constant 
$C$ depending  only on $N$, $p$, $\alpha$, 
and on the norms
$
\|b\| _{L^p  (\Omega_T) }
$,
$
\|b\| _{L^\infty\left(0,T,L^{N,\infty} (\Omega) \right)}
$, 
$ \|
u_0
\| _{L^2(\Omega)}
$, $ 
\|
H
\| _{L^1(\Omega_T)}
$, 
$\|
f
\| _{\dual}
$.
\end{proposition}
\begin{proof}
We set $w:= u/\lambda$, in such a way that $w$ solves the problem \eqref{mp4}. 
We fix $t \in (0,T)$  and we  choose   $ \varphi:= w \chi_{(0,t)} $
as a test   function for 
\eqref{mp4}
so we get
\begin{equation}\label{test1b}
\begin{split}
&
\frac 1 2
\|w(\cdot,t)\|^2_{L^2(\Omega)}
+
\int_{\Omega_{t}}
 A(x,s, \lambda  w,\nabla w) \cdot \nabla  w 
\,
\dx x \dx s 
=
\frac 1 2
\|u_0  \|^2_{L^2(\Omega)}
+
\int_0^{t} 
\left\langle
f
,
w 
\right\rangle
\dx s 
\end{split}
\end{equation}
By means of Young inequality 
we have for $0<\varepsilon <1$
\begin{equation}
\begin{split}
 \int_0^{t} 
\left\langle
f
,
w
\right\rangle
\,
\dx s 
&\le 
\varepsilon 
\int_{\Omega_t}
{|\nabla  w|^p}
 \dx x\dx s
+
C(\varepsilon)  \|f\|^{p^\prime}_{\dual} 
\end{split}
\end{equation}
By 
\eqref{1.3} we further have
\begin{equation}\label{test23}
\begin{split} 
\frac 1 2 \| w(\cdot,t) \| ^2_{L^2(\Om)}
&
+\alpha
\int_{\Omega_t}
|\nabla w|^p 
\,\dx x\dx s
 \le
\frac 12 \|  u_0\| ^2_{L^2(\Om)}
+
\varepsilon 
\int_{\Omega_t}
{|\nabla  w|^p}
\,\dx x\dx s
\\
&   
+
C(\varepsilon)  \|f\|^{p^\prime}_{\dual} 
+
\int_{\Omega_t}
\left(
 { \lambda b  |w| } 
\right)^p
\dx x\dx s
+
\int_{\Omega_t}
{H(x,s)}
\,\dx x\dx s
\end{split}
\end{equation}
Now, we provide an estimate on  $
\left\|
 {  \lambda b  |w| } 
\right\|
_{L^p(\Omega_t)}
=
\left\|
 {    b  |u| } 
\right\|
_{L^p(\Omega_t)}
$.
By Minkowski 
inequality we have
\begin{equation}\label{4.12b}
\begin{split}
\left\|
 {   b   |u| } 
\right\|
_{L^p(\Omega_t)}
&\le
\left\|
 {(\mathcal T_m b) |u| } 
\right\|
_{L^p(\Omega_t)}
+
\left\|
 {\left(   b   -  \mathcal T_m b \right)|u| } 
\right\|
_{L^p(\Omega_t)}
\end{split}
\end{equation}
Here $\mathcal T_m b$ denotes the truncation $b$ at levels $\pm m$.
We estimate separately the latter two terms. 
For $k>0$ fixed, we have 
\begin{equation}\label{4.12b1}
\begin{split}
\left\|
 {(\mathcal T_m b) |u| } 
\right\|^p
_{L^p(\Omega_t)}
\le
m^p \int_0^t  \dx s \int_{ \{|u (\cdot,s)|>k\} } |u|^p\,\dx x + 
k^p
\int_0^t  \dx s \int_{ \{|u (\cdot,s)|\le k\} } (b(x,s)  )^p\,\dx x
\end{split}
\end{equation}
In particular, we apply H\"older  inequality
\eqref{holder}, Sobolev inequality 
\eqref{Sobolevin}
slice--wise and Lemma \ref{ldlam} to get 
\begin{equation}\label{4.12b2}
\begin{split}
\int_0^t  \dx s \int_{ \{|u (\cdot,s)|>k\} } |u|^p\,\dx x 
& =  \int_0^t  \dx s \int_\Om  |u \chi_{ \{|u (\cdot,s)|>k\} } |^p\,\dx x
\\
&
\le \int_0^t  
\|
\chi_{ \{|u (\cdot,s)|>k\} }
\|^p_{L^{N,\infty}(\Omega)}
\|
u 
\|^p_{L^{p^\ast,p}(\Omega)}
\,\dx s
\\& \le
C
\left(
 {\Psi( k)}
\right)^{p/N}
\int_{\Omega_t} |\nabla u|^p\,\dx x \dx s
\end{split}
\end{equation}
for some constant $C$
only
depending on $N, p, \alpha $ and the constant  $M_0$  in 
\eqref{K0}. 
A similar combination of H\"older  inequality and  Sobolev inequality 
yields
\begin{equation}\label{4.12b3}
\begin{split}
\left\|
 {\left( b   -  \mathcal T_m b   \right)|u| } 
\right\|^p
_{L^p(\Omega_t)}
\le
S_{N,p}^p
\|
b  - \mathcal T_m  b  
\|_{L^\infty\left(0,T,L^{N,\infty} (\Omega) \right)}^p
\int_{\Omega_t} |\nabla u|^p\,\dx x \dx s
\end{split}
\end{equation}  
Inserting
\eqref{4.12b2}
and
\eqref{4.12b3}
into 
\eqref{4.12b}
we obtain
\begin{equation}\label{4.12b5}
\begin{split}
\left\|
 {b |u| } 
\right\|
_{L^p(\Omega_t)}
&\le 
\left[
C {\Psi( k)^{1/N}} +  S_{N,p}  \|
b  - \mathcal T_m  b  
\|_{L^\infty\left(0,T,L^{N,\infty} (\Omega) \right)} \right]
\|
\nabla u
\|_{L^p(\Omega_t)}
\end{split}
\end{equation}
Observe that \eqref{test23} 
and
\eqref{4.12b1}
imply
\begin{equation}\label{test27n}
\begin{split}
 \alpha^{1/p}
\|
\nabla w
\|_{L^p(\Omega_t)}
\le 
\left(
\frac 1 2
\right)^{1/p}
 \|  u_0\| ^{2/p}_{L^2(\Om)}
+
k  \|b\| _{L^p(\Om_T)}
+
C(\varepsilon,p)  \|f\|^{p^\prime/p}_{\dual} 
\\
+
\left[ \varepsilon^{1/p} +  C
  {\Psi( k)^{1/N}} +  S_{N,p}  \|
b  - \mathcal T_m  b  
\|_{L^\infty\left(0,T,L^{N,\infty} (\Omega) \right)} \right]
\|
\nabla w
\|_{L^p(\Omega_t)}
+
\|
H
\|^{1/p}_{L^1(\Omega_T)}
\end{split}
\end{equation}
We are able to reabsorb by the left hand side
by choosing properly $m$, 
$k$  and $\varepsilon$. 
For instance, it is sufficient to have  $m$ so large 
to guarantee
\begin{equation}\label{dcons}
  \|
b  - \mathcal T_m  b  
\|_{L^\infty\left(0,T,L^{N,\infty} (\Omega) \right)} < \frac {\alpha^{1/p}}{S_{N,p}}
\end{equation}
and  the existence of such value of $m$ 
is a direct consequence of \eqref{dist}.  
A proper choice of $k$ and $\varepsilon$ can be performed 
coherently with 
\eqref{dcons} and taking into account the properties of $\Psi(\cdot)$.
In particular, denoting by $C$ a constant depending   only on $N$, $p$, $\alpha$  and on 
$
\|b\| _{L^p  (\Omega_T) }
$, 
$
\mathscr D_b
$,
$ \|
u_0
\| _{L^2(\Omega)}
$, $ 
\|
H
\| _{L^1(\Omega_T)}
$, 
$\|
f
\| _{\dual}
$, 
we deduce from \eqref{test27n}
\begin{equation}\label{test30}
\begin{split}
 \|w(\cdot,t) \| ^2_{L^2(\Om)}
+
\|
\nabla w
\|^p_{L^p(\Omega_t)}
\le
C
\end{split}
\end{equation}
Taking into account \eqref{test30}  
and recalling that $\lambda \in (0,1]$,  
 the
latter estimate leads to the conclusion of the proof.
\end{proof}

\begin{remark}\label{rem11}
We point out that 
the a priori estimate \eqref{apr1bis}
is
uniform   with respect to the parameter $\lambda$. 
\end{remark}


\section{Parabolic equations with bounded coefficients}\label{bsec}
This section is devoted to the proof of the existence of a solution to problem \eqref{mp}
in the special case 
$b \in L^\infty(\Omega_T)$. We    shall use on this account
the following  version of 
 Leray--Schauder fixed point theorem 
as in 
(see  e.g. 
\cite[Theorem 11.3 pg. 280]{gt}).
\begin{theorem}
\label{LerSch2}
Let $\mathcal F$ be a compact mapping of a Banach space $X$ into itself, and suppose there exists a constant $M$ such that $\|x\|_{X}<M$ for all $x\in X$ and $\lambda\in [0,1]$ satisfying $x=\lambda\mathcal F(x).$ Then, $\mathcal F$ has a fixed point. 
\end{theorem}
We  recall that
a
continuous mapping between two Banach spaces is called compact if the images of bounded sets are precompact. 

\medskip

Accordingly,  
the main result of this section  reads as follows.
\begin{theorem} \label{ttt2}
Let assumptions 
\eqref{1.3},
\eqref{1.4} and
\eqref{1.5}
be in charge 
and $b \in L^\infty(\Omega_T)$.
Then
problem  \eqref{mp} admits  a solution.  
\end{theorem}

\begin{proof}
We let  $v \in  L^p(\Omega_T)$, and consider the  problem
\begin{equation}\label{mp2}
\left\{
\begin{array}{rl}
&  
u_t     - 
\divergenza  A(x,t,v, \nabla u)
 = 
f
\qquad\text{in $  \Omega_{T }$},
  \\
\, \\
&   u  = 0    
\qquad\text{on $\partial \Omega    \times (0,T)  $}
, \\
\, \\
& u (\cdot,0) = u_0     
\qquad\text{in $  \Omega    $}
, \\
\end{array}
\right.
\end{equation}
Problem \eqref{mp2} admits a solution by the classical theory of pseudomonotone operators (\cite{L}) and 
by the strict monotonicity of the vector field
 $$(x,t,\xi)\in \Omega_T\times \mathbb R^N\mapsto \bar A(x,t,\xi)  :=A(x,t,v,\xi)$$
such a solution is unique. 
Hence, the map $\mathcal F$ which takes $v$ to the solution $u$ is well defined and certainly acts from $L^p(\Om_T)$ into itself. 
Our goal is to determine a fixed point for $\mathcal F$, which is obviously a solution to \eqref{mp} under the 
assumption
$b \in L^\infty(\Omega_T)$. 
We want to apply Theorem
\ref{LerSch2}, so   
we need to show that 
$\mathcal F \colon L^p(\Om_T)\rightarrow L^p(\Om_T)$ is continuous, compact and the set
\[
\mathcal U : =  \left\{
u \in   
L^p(\Om_T)
\colon \,  u = \lambda  \mathcal F[u] \quad \text{for some $\lambda \in [0,1]$}
\right\}
\]
is bounded in $L^p(\Om_T)$. First observe that  the boundedness of $\mathcal U$ is  a direct consequence of 
Proposition \ref{mprop}.
%
%

\medskip

Let us prove that  $\left\{\mathcal F[v_n]\right\}_{n \in \mathbb N}$ is a precompact sequence if  $\{v_n\}_{n \in \mathbb N}$ is  a bounded sequence in $L^p(\Om_T)$. We need to show that $u_n:=\mathcal F [v_n]$ admits a subsequence strongly converging in $L^p(\Om_T)$. 
By definition of $\mathcal F$, we see that $u_n$ solves the problem
\begin{equation}\label{mp2n}
\left\{
\begin{array}{rl}
&  
\partial_t u_n     - 
\divergenza  A(x,t,v_n, \nabla u_n)
 = 
f
\qquad\text{in $  \Omega_{T }$},
  \\
\, \\
&   u_n = 0    
\qquad\text{on $\partial \Omega    \times (0,T)  $}
, \\
\, \\
& u_n (\cdot,0) = u_0     
\qquad\text{in $  \Omega    $}
, \\
\end{array}
\right.
\end{equation}
By testing the equation in \eqref{mp2n} by $u_n$ itself and arguing similarly as before, we see that 
\begin{equation}\label{apr50}
\begin{split}
\sup_{0<t<T}
&  \int_\Omega |\spp  (\cdot,t) |^2\, \dx x
 + 
\int_{\Omega_T}
| \nabla \spp   |^p
\, \dx x \dx t
\\
& \le
C  \left[
\|
u_0
\|^2_{L^2(\Omega)}
+
\|
H
\| _{L^1(\Omega_T)}
 +
 \|b\|^p_{L^\infty(\Omega_T)}
\|v_n\|^p_{L^p(\Omega_T)}+
\|
f
\|^{p^\prime}_{\dual}
\right]
\end{split}
\end{equation}
So in particular $\{|\nabla u_n|\}_{n\in \mathbb N}$ is bounded in $L^p(\Om_T)$.
Using 
the equation in 
\eqref{mp2} we see that 
$\{\partial_t \spp\}_{n \in \mathbb N}$ is  bounded 
in $\dual$ and so $\{\spp\}_{n \in \mathbb N}$ strongly converges in $L^p(\Om_T)$
as a
direct consequence
of the Aubin--Lions lemma. 

\medskip

Let us prove the continuity of $\mathcal F$. 
 Let $\{v_n\}_{n \in \mathbb N}$ be a strongly converging sequence  in $L^p(\Om_T)$, say  
\[
v_n
\rightarrow
v
\qquad \text{in $L^p(\Om_T)$ strongly}
\]
We set 
$u_n:=\mathcal F [v_n]$. 
We already know that $\{u_n\}_{n \in \mathbb N}$  is compact sequence
in $L^p(\Om_T)$
and also that estimate 
\eqref{apr50}
holds true. So, there exists $u \in W_p(0,T)$ such that
\begin{align}
 u_n
& \rightarrow
 u 
\qquad \text{strongly in $L^p(\Om_T)$}
 \label{70d}  \\
\nabla \spp & \rightharpoonup \nabla u \quad\text{weakly in $L^p\left(\Omega_T,\mathbb R^N\right)$} \label{71d}  \\
\spp & \rightharpoonup^\ast  u  \quad\text{weakly$^{\ast}$ in $L^\infty(0,T;L^2(\Omega))$} 
\label{72d}\\
\partial _ t \spp & \rightharpoonup \partial _ t u  \quad\text{weakly  in $\dual$} 
\label{73d}
\end{align}
Observe further that $u \in C^0\left([0,T],L^2(\Om)\right)$ with $u(\cdot,0)=u_0$. It is a direct consequence of the inclusion of Lemma \ref{incl}, the boundednes of $\{u_n\}_{n \in \mathbb N} $
in $W_p(0,T)$
 and of the convergence $u_n\rightharpoonup u$   weakly in $L^2(\Om)$ for all $t \in[0,T]$.
In order to 
prove the
continuity of $\mathcal F$, 
 we need to show that 
\begin{equation}\label{ffff}
u =\mathcal F [v ]
\end{equation}
Again, we know that $u_n$ solves \eqref{mp2n}, namely for every 
$\varphi \in C^\infty(\Om_T)$ with  support contained in $[0,T)\times\Om$ we have 
\begin{equation}\label{dsn}
\int_0^T
\langle \partial_t u_n,\varphi \rangle
\dx s
+
\int_0^T \int_\Omega A(x,s,v_n,\nabla u_n)\cdot \nabla \varphi \dx x \dx s
=
\int_0^T 
 \langle f,  \varphi \rangle 
\,\dx s
\end{equation} 
If we choose $u_n - u$ as a test function in the above identity we have
\begin{equation}\label{test51}
\begin{split}
\frac 1 2 \|u_n(t)-u(t)\|^2_{L^2(\Om)} 
& +
\int_0^T
\langle \partial_t u ,u_n-u \rangle
\dx s
\\
& + \int_0^T \int_\Omega A(x,s,v_n,\nabla u_n)\cdot \nabla (u_n-u) \dx x \dx s
= \int_0^T 
 \langle f,  u_n-u \rangle 
\,\dx s
\end{split}
\end{equation} 
Now, it is clear that
\begin{equation}\label{test52}
\begin{split}
\limsup_{n\rightarrow \infty}  \int_0^T \int_\Omega A(x,t,v_n,\nabla u_n )\cdot \nabla (u_n-u) \dx x \dx t \le 0
\end{split}
\end{equation} 
From the boundedness of $b$, the sequence  
$
\{
A(x,s,v_n,\nabla u)
\}_{n\in \mathbb N}
$ strongly converges
in $L^{p^\prime}(\Om_T)$ and so
\[
\lim_{n\rightarrow \infty}  \int_0^T \int_\Omega A(x,t,v_n,\nabla u )\cdot \nabla (u_n-u) \dx x \dx t = 0
\]
The latter estimate, the  strict monotonicity of $A$
and 
%
 \eqref{test52}  give us
\begin{equation}\label{test52b}
\begin{split}
\lim_{n\rightarrow \infty}  \int_0^T \int_\Omega A(x,t,v_n,\nabla u_n )\cdot \nabla (u_n-u) \dx x \dx t = 0
\end{split}
\end{equation}

It is clear that 
$\{
A(x,s,v_n,\nabla u_n)
\}_{n\in \mathbb N}$ is bounded in $L^{p^\prime}(\Om_T)$ and so it weakly
converges  in 
$L^{p^\prime}(\Om_T)$ to some $\tilde A$. We use the Minty trick to recover that $\tilde A (x,t)=A(x,t,v(x,t),\nabla u (x,t))$ a.e. in $\Om_T$. Namely, let $\eta \in L^p(\Om_T,\mathbb R^N)$. Observe that
\begin{equation}\label{dsl22}
\begin{split}
0 \le \int_{\Om_T}  & \left[
A(x,t,v_n,\nabla u_n)-A(x,t,v_n,\eta)
\right]\cdot (\nabla u_n - \eta)\dx x \dx t
\\& =\int_{\Om_T}   \left[
A(x,t,v_n,\nabla u_n)-A(x,t,v_n,\nabla u )
\right]\cdot (\nabla u_n - \nabla u )\dx x \dx t
\\ & \qquad + 
\int_{\Om_T}   \left[
A(x,t,v_n,\nabla u )-A(x,t,v_n,\eta  )
\right]\cdot (\nabla u_n - \nabla u )\dx x \dx t
\\ & \qquad + 
\int_{\Om_T}   \left[
A(x,t,v_n,\nabla u_n )-A(x,t,v_n,\eta  )
\right]\cdot (\nabla u - \eta   )\dx x \dx t
\end{split}
\end{equation} 
Passing to the limit, we get
\begin{equation}\label{dsl223}
\begin{split}
0 \le   
\int_{\Om_T}   \left[\tilde A(x,t) -A(x,t,v ,\eta  )
\right]\cdot (\nabla u - \eta   )\dx x \dx t
\end{split}
\end{equation} 
We choose 
$\eta:= \nabla u - \lambda \psi$
in \eqref{dsl223}
 where $\psi \in L^p(\Om_T,\mathbb R^N)$ and $\lambda \in \mathbb R$.
Then 
\begin{equation}\label{dsl2233}
\begin{split}
   \lambda 
\int_{\Om_T}   \left[\tilde A(x,t) -A(x,t,v , \nabla u -  \lambda \psi  )
\right]\cdot   \psi \dx x \dx t \ge 0 
\end{split}
\end{equation} 
If we assume that $\lambda >0$, we divide by $\lambda$ itself and then letting $\lambda \rightarrow 0^+$
we have 
\begin{equation}\label{dsl226}
\begin{split}
\int_{\Om_T}   \left[\tilde A(x,t) -A(x,t,v , \nabla u    )
\right]\cdot   \psi \dx x \dx t \ge 0 
\end{split}
\end{equation} 
Arguing similarly if $\lambda <0$ we get the opposite inequality than 
\eqref{dsl226}, so we conclude that for every 
$\psi \in L^p(\Om_T,\mathbb R^N)$
\begin{equation}\label{dsl228}
\begin{split}
\int_{\Om_T}   \left[\tilde A(x,t) -A(x,t,v , \nabla u    )
\right]\cdot   \psi \dx x \dx t = 0 
\end{split}
\end{equation} 
i.e. $\tilde A (x,t)=A(x,t,v(x,t),\nabla u (x,t))$ a.e. in $\Om_T$.

\medskip

We are in position to 
pass to the limit in \eqref{dsn}. Therefore  
\begin{equation}\label{dsl}
-
\int_0^T  \int_\Om 
u   \partial _t    \varphi \dx s
+
\int_0^T \int_\Omega A(x,s,v ,\nabla u)\cdot \nabla \varphi \,dxds
+ \int_\Omega u_0 \varphi (x,0)\,\dx x
=
\int_0^T 
 \langle f,  \varphi \rangle 
\,\dx s
\end{equation} 
that is \eqref{ffff} holds. 
\end{proof}

\section{Proof of the existence result via  approximation scheme}\label{sec3}
\begin{proof}[Proof of Theorem \ref{ttt}]
Let $n \in \mathbb N$. 
We introduce the following initial--boundary 
value
problem
\begin{equation}\label{p1}
\left\{
\begin{array}{rl}
&  
\partial_t \spp   - 
\divergenza 
A(x,t,\theta_n \spp, \nabla  \spp   )
= 
f
\qquad\text{in $  \Omega_{T }$},
  \\
\, \\
&   \spp = 0    
\qquad\text{on $\partial \Omega    \times (0,T)  $}
, \\
\, \\
& \spp (\cdot,0) = u_0     
\qquad\text{in $  \Omega    $}
, \\
\end{array}
\right.
\end{equation}
where \[
\theta_n (x,t):= \frac {\mathcal T_n b(x,t)  }{b(x,t)} \qquad \text{for $(x,t)\in \Om_T$} 
\] and
\[ (x,t,u,\xi)\in \Omega_T\times \mathbb R \times \mathbb R^N\mapsto
A_n(x,t,u, \xi)
:=
A(x,t,\theta_n u, \xi   )
\]
To achieve the proof of Theorem \ref{ttt} we need to pass to the limit in \eqref{p1}.
%
%
%
%
%
%
%
The results of  Section  \ref{bsec}
provides the existence of   solution 
$$
\spp \in  \parspcont \cap \pspace
$$
to problem \eqref{p1}. 
In fact,      
$A_n=A_n(x,t,u,\xi)$  satisfies
\eqref{1.3},
\eqref{1.4} and
\eqref{1.5}
with 
$\mathcal T_n b $ in place of 
$ b $
and $\mathcal T_n b  $ belongs to $L^\infty(\Omega_T)$ for each fixed $n\in \mathbb N$.
Since 
 $\mathcal T_n b \le b$ in $\Omega_T$ for every $n \in \mathbb N$, by
 Proposition \ref{mprop},  there exists  a positive constant independent of $n$ 
%
such that 
the following   estimate for a solution to problem \eqref{p1}  holds
\begin{equation}\label{apr1ne}
\begin{split}
\sup_{0<t<T}
 \int_\Omega |\spp  (\cdot,t) |^2\, \dx x
+ 
\int_{\Omega_T}
| \nabla \spp   |^p
\, \dx x \dx t
 \le
C  
\end{split}
\end{equation}
Hence,   
there exists
$
u  \in \parspbound \cap \pspace   
$
such that
\begin{align}
\spp   & \rightharpoonup u \quad\text{weakly in $L^p(\Omega_T)$} \label{70}  \\
\nabla \spp & \rightharpoonup \nabla u \quad\text{weakly in $L^p\left(\Omega_T,\mathbb R^N\right)$} \label{71}  \\
\spp & \rightharpoonup^\ast  u  \quad\text{weakly$^{\ast}$ in $L^\infty(0,T;L^2(\Omega))$} 
\label{72}
\end{align}
as $n \rightarrow \infty$. 
With the aid of 
the equation in 
\eqref{p1}, we obtain 
a
uniform bound  for the norm of the time derivative of $u_n$ in 
$\dual$. 
Therefore, by
 Aubin--Lions
 lemma we have 
\begin{align}
\spp   & \rightarrow u \quad\text{strongly in $L^p(\Omega_T)$ and a.e. in $\Omega_T$} \label{70bis}  \\
\partial_t \spp & \rightharpoonup \partial_t u \quad\text{weakly in $\dual$} \label{71bis} 
\end{align}
Note also
that  $
u \in \parspcont
$  and $u(\cdot,0)=u_0$. As before, this is a consequence of  Lemma \ref{incl}, 
 the boundednes of $\{u_n\}_{n \in \mathbb N} $
in $W_p(0,T)$
 and of the convergence $u_n\rightharpoonup u$   weakly in $L^2(\Om)$ for all $t \in[0,T]$.

\medskip

In the last stage of our proof we want to pass   to the limit in \eqref{p1}. 
For  $z \in \mathbb R$, we set  
$$\gamma(z):=\int_0^z \frac {\dx \zeta} {1+ |\zeta|^{p}}$$ 
Obviously, $\gamma \in C^1(\mathbb R)$, $\gamma$ is odd, $|\gamma (z)|\leqslant |z|$ and $0 \leqslant \gamma^\prime(z) \leqslant 1$ for all $z \in \mathbb R$. In particular,
$\gamma$ is Lipschitz continuous in the whole of $\mathbb R$ and therefore 
$\gamma(u_n-u)\in \pspace$.
Moreover, since $\gamma(0)=0$ we deduce from \eqref{71}
and \eqref{70bis} 
\begin{align}
\label{23}
\gamma(u_n-u) & \rightharpoonup 0    \qquad \text{in $\pspace$ weakly}\\
\label{23ter}
\gamma(u_n-u)  & \rightarrow  0    \qquad \text{in $L^p(\Om_T)$ strongly and a.e. in $\Om_T$}
\end{align}
We observe that $\gamma(u_n-u)|_{t=0} = 0$.
Testing equation \eqref{p1} by  $\gamma(u_n-u)$ we get
\begin{equation}\label{test1}
\begin{split}
&
\int_\Omega
\Gamma(
 \spp-u  )|_{t=T}
  \dx x
+
\int_{\Omega_T}
 A_n(x,t, \spp,\nabla \spp) \cdot \nabla \gamma(u_n-u)
\,
\dx x \dx t 
\\
&\qquad =
-\int_0^T
\langle
\partial _ t u 
,
\gamma(
 \spp-u  )
\rangle \dx t
+
\int_0^T 
\left\langle
f
,
  \gamma(u_n-u)
\right\rangle
\,
\dx t
\end{split}
\end{equation}
where  
$\Gamma(z):=\int_0^z \gamma(\zeta)\dx\zeta$ for $z \in \mathbb R$.
Moreover, we have
$\nabla \gamma(u_n-u) = \gamma^\prime(u_n-u) (\nabla u_n - \nabla u)$.
%
From
\eqref{23}
it follows that
\begin{align}
\lim_{n\rightarrow \infty}
\int_0^T
\langle
\partial _ t u 
,
\gamma(
 \spp-u  )
\rangle \dx t & = 0 
\\
\lim_{n\rightarrow \infty}
\int_0^T 
\left\langle
f
,
  \gamma(u_n-u)
\right\rangle
\,
\dx t & =  0 
\end{align}
Then by \eqref{test1}
we   have
\begin{equation}\label{24}
\limsup_{n\rightarrow \infty}
\int_{\Omega_T} A_n(x,t,u_n,\nabla u_n) \nabla \gamma(u_n-u)\dx x \dx t \le   0\,.
\end{equation} 
We claim that
\begin{equation}\label{26}
\lim_{n\rightarrow \infty} \int_{\Omega_T} A_n(x,t,u_n,\nabla u) \nabla \gamma(u_n-u) \dx x \dx t = 0\,.
\end{equation}   
In view of \eqref{71}, to get \eqref{26}
it suffices to show that
\begin{equation}\label{compattezza}
\gamma'(u_n-u)\,A_n(x,t,u_n,\nabla u)= \frac { A_n(x,t,u_n,\nabla u)  }{1+|u_n-u|^{p}}\qquad \text{is compact in }L^{p'}(\Om_T,\mathbb R^N)\,.
\end{equation}
Preliminarily, we observe that  combining  \eqref{70bis} with the property that $\theta_n\rightarrow 1$ as $n\rightarrow \infty$, we have
\[
\frac { A_n(x,t,u_n,\nabla u)  }{1+|u_n-u|^{p}} \rightarrow  
A (x,t,u ,\nabla u) \qquad \text{a.e.\ in }\Omega\,.
\]
On the other hand, we see that
\[
\left|\frac { A_n(x,t,u_n,\nabla u)  }{1+|u_n-u|^{p}}\right|^{p'}\leqslant C(\beta,p )
\left[|\nabla u|^p+ K^{p^\prime} + (b|u|)^{p}+ b^{p} \right]
\] 
%
Now, from the monotonicity 
condition \eqref{1.4},
\eqref{24} 
and \eqref{26} we get
\begin{equation}\label{27}
\lim_{n\rightarrow \infty}
\int_{\Omega_T} \left[ A_n(x,t,u_n,\nabla u_n) -  A_n(x,t,u_n,\nabla u) \right]\nabla \gamma(u_n-u)  \dx x \dx t= 0\,.
\end{equation}
As the integrand is nonnegative, we have (up to a subsequence)
\[\left[ A_n(x,t,u_n,\nabla u_n) -  A_n(x,t,u_n,\nabla u) \right]\nabla \gamma(u_n-u)\to 0 \qquad \text{a.e.\ in $\Omega_T$.}\]
Moreover, since $\gamma'(u_n-u)\to1$ a.e.\ in $\Omega_T$, the above in turn implies
\begin{equation}\label{202005021}
\left[ A_n(x,t,u_n,\nabla u_n) -  A_n(x,t,u_n,\nabla u) \right]\,(\nabla u_n-\nabla u)\to 0
\qquad \text{a.e.\ in $\Omega_T$.}
\end{equation}
Arguing as in the proof of \cite[Lemma~3.3]{LL}, we see that
\begin{equation}\label{29.4}
\nabla u_n  \rightarrow  \nabla u \qquad \text{a.e.\ in $\Omega_T$}
\end{equation}
and
\begin{equation}\label{29.6}
A_n (x,t,u_n,\nabla u_n)
  \rightharpoonup
A  (x,t,u ,\nabla u )
 \qquad \text{in 
$L^{p^\prime} (\Omega_T,\mathbb R^N)$ weakly}
\end{equation}
and we conclude that $u$ solves 
the original problem~\eqref{mp}.
\end{proof}
We conclude this section providing an example which shows that 
assumption \eqref{dist} in general cannot be dropped, even for linear problems. 

\section{Asymptotic behavior}

This section is devoted to study the time behavior of a solution to problem \eqref{mp}. Through this section we assume that  \eqref{1.3}, 
\eqref{1.5} 
and \eqref{b} 
 are in charge  
and that
condition \eqref{dist} is satisfied. 
%
%

\medskip

For convenience we will
 denote by 
\begin{equation}\label{gdef}
g(t):= \|f(t)\|^{p^\prime}_{W^{-1,p^\prime}(\Omega)} + \|H(t)\|^{p}_{L^{ p }(\Omega)}+
 \|b(t)\|^{p}_{L^{ N,\infty }(\Omega)}
\end{equation}
for a.e. $t\in[0,T]$.
The first result of the present  section is the following
\begin{theorem}\label{asyt1}
Under the above assumptions, any solution $u$
to
problem \eqref{mp} satisfies
 for any $t \in  [0,T]$ the following estimate
\begin{equation}
\|u(t)\|^2_{L^2(\Omega)}   \le x (t)
\end{equation}
where $x (t)$ is the unique  solution of the problem 
\begin{equation} \label{c4}
\left\{
\begin{array}{rl}
&  
\dot{x} (t)  = - M_1  
(x (t) )^{p/2} +  C g(t)     
  \\
\, \\
&   x(0)= \|u_0\|^2_{L^2(\Omega)}    
\end{array}
\right. \qquad \mbox{ if }p>2
\end{equation}
or 
\begin{equation} \label{c5}
\left\{
\begin{array}{rl}
&  
\dot{x}(t) = - M_2  x(t)  +  C g(t)     
  \\
\, \\
&   x(0)= \|u_0\|^2_{L^2(\Omega)}    
\end{array}
\right. \qquad \mbox{ if }p\leq 2
\end{equation}

Here $C>0$ and $M_i>0, \,i=1,2,$ are  positive constants depending 
on $N$, $p$, $\alpha$, 
$\beta$, $ |\Omega |$ and
$
\mathscr  D_b.
$
\end{theorem}

\begin{proof}
We test the equation in \eqref{mp} by the solution $u$ itself. We get the following energy equality
\[
\frac 1 2
\|u(t)\|^2_{L^2(\Om)}   + \int_{\Om_t} A(x,s,u,\nabla u)\cdot \nabla u \dx x \dx s = \frac 1 2  \|u_0\|^2_{L^2(\Om)} 
+
\int_0^t \langle f , u \rangle \dx s
\]
We write 
down such equality first for $t=t_1$ and subsequently for $t=t_2$, we  subtract 
the relations obtained in this way and we deduce
\[
\frac 1 2 \|u(t_2)\|^2_{L^2(\Om)} - \frac 1 2  \|u(t_1)\|^2_{L^2(\Om)} +\int_{t_1}^{t_2} \int_\Om A(x,s,u,\nabla u)\cdot \nabla u \dx x \dx s =
\int_{t_1}^{t_2} \langle f , u \rangle \dx s
\]
Using \eqref{1.3} we get
\[
\begin{split}
\frac 1 2\|u(t_2)\|^2_{L^2(\Om)} -\frac 1 2  &  \|u(t_1)\|^2_{L^2(\Om) }   +\alpha \int_{t_1}^{t_2} \int_\Om |\nabla u|^p   \dx x \dx s \\\\ & \le
\int_{t_1}^{t_2} \int_\Om (b(x,s)  |u|)^p   \dx x \dx s 
+
\int_{\Om \times (t_1,t_2) } H(x,s)    \dx x \dx s 
+
\int_{t_1}^{t_2} \langle f , u \rangle \dx s
\end{split}
\]
By Young inequality the latter inequality implies 
\begin{equation}\label{yas}
\begin{split}
&
\frac 1 2\|u(t_2)\|^2_{L^2(\Om)}  -  \frac 1 2     \|u(t_1)\|^2_{L^2(\Om) }
 +\alpha 
\|
\nabla u
\|^p_{L^p(\Om \times (t_1,t_2) )}
 \\  & \le
\left\|
b|u|
\right\|^p_{L^p(\Om \times (t_1,t_2) )}
+
\left\|
H
\right\|_{L^1(\Om \times (t_1,t_2) )} 
+
C(\varepsilon)\|
f
\|^{p^\prime}_{L^{p^\prime} \left( t_1,t_2. W^{-1,p^\prime}(\Om)\right) }
+
\varepsilon
\|
\nabla u
\|^p_{L^p(\Om \times (t_1,t_2) )}
\end{split}
\end{equation}
From 
Proposition \ref{mprop} (applied for $\lambda=1$)
 it is clear that $\|\nabla u\|_{L^p(\Om_T)}$ is controlled by a quantity depending only on the data, so 
H\"older and  Sobolev inequalities applied slice--wise give us 
$$
\left\|
b|u|
\right\|^p_{L^p(\Om \times (t_1,t_2) )} \le C \int_{t_1}^{t_2} \|b(\cdot,s)\|_{L^{N,\infty} (\Om)}\dx s
$$
Taking into accout all the above relations,  finally
we get
\begin{equation}\label{5.5}
\begin{split}
\|u(t_2)\|^2_{L^2(\Om)} - &  \|u(t_1)\|^2_{L^2(\Om) }   + C  \int_{t_1}^{t_2} \int_\Om |\nabla u|^p   \dx x \dx s \le C_0
 \int_{t_1}^{t_2  }  g(s) \dx s   
\end{split}
\end{equation}
for some constants $C,C_0$ depending on the data.
By 
Sobolev inequality   
we have
\begin{equation}\label{5.1b}
\begin{split}
\|u(t_2)\|^2_{L^2(\Om)} - &  \|u(t_1)\|^2_{L^2(\Om) }   + C  \int_{ t_1 } ^{t_2  } 
\| u\|^p_{L^{p^\ast}(\Omega)}   \dx x \dx s \le
C_0
 \int_{t_1}^{t_2  }  g(s) \dx s    
\end{split}
\end{equation}
Since
$p>2N/(N+2)$ is equivalent to
$p^\ast>2$, then  
\[
\left(\average \Om {|u(\cdot,s)|^2} x \right)^{p^\ast/2} \le 
 \average \Om {|u(\cdot,s)|^{p^\ast}} x  
\]
for a.e. $s\in (t_1,t_2)$.
Hence, from  
\eqref{5.1b}
\begin{equation}\label{5.1bb}
\begin{split}
\|u(t_2)\|^2_{L^2(\Om)} - &  \|u(t_1)\|^2_{L^2(\Om) }   + 
C
\int_{t_1}^{t_2}
\|u(s)\|^{p }_{L^2(\Om)}
\dx s  
\le 
C_0
 \int_{t_1}^{t_2  }  g(s) \dx s   
\end{split}
\end{equation}
If 
$2N/(N+2) <p \le 2$, then by
\eqref{5.5} we have
\begin{equation}\label{5.5b}
\begin{split}
\|u(t )\|^2_{L^2(\Om)}       \le
\Lambda_0:=
\left(
  \|u_0\|^2_{L^2(\Om) }
+
C_0
 \int_{0}^{T}  g(s) \dx s   
\right)
\end{split}
\end{equation}
In turn, we have
\[
\|u(t )\|^{ p }_{L^2(\Om)}   
=
\frac
{\|u(t )\|^2_{L^2(\Om)}   }
{
\|u(t )\|^{2-p }_{L^2(\Om)}   
}
\ge 
\frac
{\|u(t )\|^2_{L^2(\Om)}   }
{
\Lambda_0^{\frac{2-p } 2}
}
\]
and therefore from  \eqref{5.1bb}
\begin{equation}\label{5.1t}
\begin{split}
\|u(t_2)\|^2_{L^2(\Om)} - &  \|u(t_1)\|^2_{L^2(\Om) }   + 
C
\int_{t_1}^{t_2}
\|u(s)\|^2 _{L^2(\Om)}
\dx s  
\le 
C_0
 \int_{t_1}^{t_2  }  g(s) \dx s   
\end{split}
\end{equation}
Finally, our claimed result follows directly by Lemma \ref{lCau1} where we choose
$t_0=0$,
$\phi(t):=\|u(t)\|^2_{L^2(\Om)}$
and
$\psi(t,y)\equiv\psi(y):= M |y|^{1+\nu}$ for $M$ positive constant and $\nu=p/2 - 1$ for $p >2$ and $\nu=0$ for $2N/(N+2)<p\le 2$.
%
%
\end{proof}
As a byproduct of previous result, we are able to show that the $L^2(\Om)$--norm of any solution
to problem \eqref{mp}
 decays as an explicit negative power of the time variable and exponentially fast respectively in case
$p>2$ and $2N/(N+2)<p \le 2$
by using Lemma \ref{lCau2}.   
\begin{corollary}\label{finalcor}
If  $u$ is a solution 
to
problem \eqref{mp}, then
 for any 
$t \in  [0,T]$ 
\begin{align}
\label{5.12}
\|u(t)\|^2_{L^2(\Omega)} & \le   \frac
{\|u_0\|^2_{L^2(\Omega)}}
{\left[
1+\left(\frac p 2 -1\right) M_1
 {\|u_0\|^{p   -2}_{L^2(\Omega)}}  
  t  
\right]^{\frac 2 {p-2} } }
+C_0  \int_{0}^{t} g(s) \, ds
\qquad\text{if $p>2$}
\\
\label{5.13}
\|u(t)\|^2_{L^2(\Omega)} &  \le \|u_0\|^2_{L^2(\Omega)} e^{-M_2 t } +C_0
\int_0^t e^{-M_2(t-s)} g(s) \dx s
\qquad\text{if $\frac{2N}{N+2}<p\le2$}
\end{align}
Here $C_0>0$ and $M_i>0, \,i=1,2,$ are  positive constants depending 
on $N$, $p$, $\alpha$, 
$ |\Omega |$ and
$
\mathscr  D_b.
$
%
%
%
%
\end{corollary}
\begin{proof}
We set again
$\phi(t):=\|u(t)\|^2_{L^2(\Om)}$. It had been already observed in 
\cite{Moscariello-Frankowska}
that if $p>2$ the function $x(\cdot)$ of Lemma \ref{lCau1} satisfies
\[x(t)\le y(t):=\frac{\phi(0)}
{ \left[
1+\nu M
\phi(0) ^\nu
 t  
\right]^{1/\nu} } + C_0 \int_{ 0}^t g(s)\dx s \]
choosing $\nu=p/2 -1$. 
The estimate \eqref{5.13}
follows by solving the Cauchy problem 
\eqref{c5}.
\end{proof}
Now we are able to describe the asymptotic behavior of a solution. 
Assume that conditions
\eqref{1.3} 
and \eqref{1.5},  hold 
and 
 \eqref{dist} hold  true 
for $t\in(0, \infty)$. 
\begin{proposition}
Let 
$ u \in C_{\rm loc}^0 \left(
[0,\infty),L^2(\Om)
\right)\cap L_{\rm loc}^p\left(
0,\infty,W^{1,p}_0 (\Omega)
\right)$. Then 
\begin{align}
\label{5.14}
\|u(t)\|^2_{L^2(\Omega)} &  \le
\left[
\left(\frac p 2 -1\right) M_1
\right]^{-\frac 2 {p-2}}
t ^{-2/(p -2)}  
+  C_0 \int_{t/2}^{t} g(s) \, ds
\quad
  \text{if $p>2$} 
\\
\label{5.15}
\|u(t)\|^2_{L^2(\Omega)}  & \le \left( {\|u_0\|^2_{L^2(\Omega)}} 
+C_0 \|g\|_{L^1([0,\infty))}
\right)
e^{- \frac {M_2} 2  t   } 
+C_0
\int_{t/2}^t e^{-\frac{M_2} 2 (t-s)} g(s) \dx s
%
%
%
\quad
  \text{if $\frac{2N}{N+2}<p\le2$} 
\end{align} 
%
%
provided $g(t)$ is integrable in $[0,\infty)$.
Here $C_0>0$ and $M_i>0, \,i=1,2,$ are  positive constants depending 
on $N$, $p$, $\alpha$, 
$ |\Omega |$ and
$
\mathscr  D_b.
$
%
%
\end{proposition}
The proof follows the same lines as in Proposition 3.1 and Proposition 3.2 of 
\cite{Moscariello-Porzio}. 
\begin{remark}  
We remark that when $p>2$ previous proposition provides universal estimate  on time, in the sense that the estimate does not depend on the initial datum $u_0$.
\end{remark}  

\begin{remark}   
The statement of Theorem 
\ref{asyt1} 
 improves the behavior
in time of a solution to problem \eqref{mp} 
known so far  \cite{Moscariello-Frankowska,Moscariello-Porzio}   for $p$--Laplace operator and see  also \cite{Far,FarMos} for $p=2$. 
\end{remark}

\end{document}